\documentclass[12pt,oneside]{amsart}
\usepackage{amsmath,amssymb,latexsym,soul,cite,mathrsfs}
\usepackage{enumitem}
\usepackage{color,enumitem,graphicx}
\usepackage[colorlinks=true,urlcolor=blue,
citecolor=red,linkcolor=blue,linktocpage,pdfpagelabels,
bookmarksnumbered,bookmarksopen]{hyperref}
\usepackage[english]{babel}
\usepackage{lineno}
\usepackage[left=2.4cm,right=2.4cm,top=2.5cm,bottom=2.4cm]{geometry}

\usepackage[hyperpageref]{backref}
\numberwithin{equation}{section}
\newtheorem{theorem}{Theorem}[section]
\newtheorem{lemma}[theorem]{Lemma}

\newtheorem{remark}[theorem]{Remark}

\title[Gradient term for the Pucci extremal operators]{On a quadratic gradient natural term for the Pucci extremal operators}

\author[J.F.\ de Oliveira]{Jos\'{e} Francisco de Oliveira}
\author[J.M. do \'{O}]{Jo\~{a}o Marcos do \'{O}}
\author[P. Ubilla]{Pedro Ubilla}
\author[A. Macedo]{Abiel Macedo}

\address[J.F.\ de Oliveira]{
\newline\indent Department of Mathematics
	\newline\indent 
	Federal University of Piau\'{i}
	\newline\indent
	64049-550 Teresina, PI, Brazil}
	\email{\href{mailto:jfoliveira@ufpi.edu.br}{jfoliveira@ufpi.edu.br}}	
	
\address[J.M. do \'{O}]{\newline\indent Department of Mathematics
\newline\indent 
Federal University of Para\'{\i}ba
\newline\indent
58051-900 Jo\~{a}o Pessoa, PB, Brazil}
\email{\href{mailto:jmbo@mat.ufpb.br}{jmbo@mat.ufpb.br}}

\address[P. Ubilla]{\newline\indent Departamento de Matematica
\newline\indent 
Universidad de Santiago de Chile
\newline\indent
Casilla 307, Correo 2, Santiago, Chile}
\email{\href{mailto:pedro.ubilla@usach.cl}{pedro.ubilla@usach.cl}}

\address[A. Macedo]{\newline\indent Instituto de Matemática e Estatística
\newline\indent Universidade Federal de Goi\'{a}s,
\newline\indent
74001-970  Goi\^{a}nia, Brazil}
\email{\href{mailto:abielcosta@ufg.br}{abielcosta@ufg.br}}


\begin{document}
\maketitle
\begin{abstract}
We introduce a quadratic gradient type term for the Pucci extremal operators. Our analysis demonstrates that this proposed term extends the classical quadratic gradient term associated with the Laplace equation, and we investigate the impact of the Kazdan-Kramer transformation. As an application, we explore the existence, non-existence, uniqueness, Liouville-type results, and asymptotic behavior of solutions for the new class of Pucci equations under various conditions on both nonlinearity and domains.
\end{abstract}
\section{Introduction}
\noindent Let $0<\lambda\le \Lambda$ and $M$ be a symmetric $n\times n$ matrix. Denote $\mathcal{M}^{\pm}(M)=\mathcal{M}^{\pm}_{\lambda,\Lambda}(M)$  the matrix operators be given by
\begin{align*}
    \mathcal{M}^{+}_{\lambda,\Lambda}(M)=\Lambda\sum_{e_i>0}e_i+\lambda\sum_{e_i<0}e_i\;\;\mbox{and}\;\;  \mathcal{M}^{-}_{\lambda,\Lambda}(M)=\lambda\sum_{e_i>0}e_i+\Lambda\sum_{e_i<0}e_i
\end{align*}
where $e_i=e_i(M)$, $i=1,\cdots, n$ are the eigenvalues of $M$. The Pucci extremal operators \cite{Pucci1,Cabre} are defined by  $ \mathcal{M}^{+}_{\lambda,\Lambda}(D^2u)$ and  $\mathcal{M}^{-}_{\lambda,\Lambda}(D^2 u)$, where $u$ is a scalar function defined on an open set $\Omega\subset\mathbb{R}^n$ and $D^2u$ denotes the Hessian matrix of $u$. These operators frequently arise in the context of optimal stochastic control problems and can be applied in financial mathematics, see for instance \cite{Avella,Lions}.  There are several works \cite{Arm-Sira,Pucci1,Cabre,Felmer,Esteban, FelmerProc,FelmerTransations} that have addressed to investigate  the Pucci equations
\begin{equation}
	\begin{cases}\label{Pucci-eq1}
	\mathcal{M}^{\pm}_{\lambda,\Lambda}(D^2 u)+f(x,u)=0	 & \mbox{in} \;\; \Omega\\
		 u=0  & \mbox{on}\;\;\partial \Omega,	
	\end{cases}
\end{equation}
where $f:\overline{\Omega}\times [0,\infty)\to \mathbb R$ is a continuous function. We recommend \cite{FelmerProc} for a review about some results on equations involving the Pucci extremal operators.

If $\lambda=\Lambda=1$,  then $\mathcal{M}^{\pm}_{\lambda,\Lambda}(D^2 u)=\mathrm{trace}(D^2u)=\Delta u$ is the Laplace operator. In this case, motivated by both non-existence results from J. Serrin \cite{Serrin} and an invariant criteria proposed by Kazdan and Kramer \cite{kazdan}, the following Laplace equation with quadratic growth in the gradient has been extensively studied
\begin{equation}\label{KK-problem}
    \left\{\begin{aligned}
   &\Delta u+|\nabla u|^2+ f(x,u)=0 & \;\;&\mbox{in}&\;\;\Omega\\
& u=0& \;\;&\mbox{on}&\;\;\partial\Omega.
    \end{aligned}
    \right.
\end{equation}
Equations like \eqref{KK-problem} have the advantage  of being able to be transformable into a new  Laplace equation without the quadratic term  $|\nabla u|^2$. Indeed, by using the Kazdan-Kramer change of variables $v=e^u-1$, $u>0$  the problem ~(\ref{KK-problem}) becomes equivalent to the following
	\begin{equation}\label{KK-problem2}
	\begin{cases}
		\Delta v+ (1+v)f(x,\log(1+v))=0 & \mbox{in}\;\; \Omega\\
		\ v=0 & \mbox{on}\;\; \partial \Omega. 
	\end{cases}
	\end{equation}
This approach has been successfully applied by many authors and several results have been achieved, see \cite{kazdan,MET2,JEANJEAN1,JEANJEAN2, MET1,Arcoya} and references therein. Extensions for the $p$-Laplace equation can be found in \cite{ubilla, Adimurthi1}, and for the $k$-Hessian equation in \cite{Nosso,Monge}.

In this paper, we investigate the existence of nonnegative solutions to the following Pucci equations  
\begin{equation}\nonumber
\hypertarget{Pmais}{(\mathcal{P}^{\pm}_{g})}\quad\quad\quad	\begin{cases}\label{Pucci-full}
	\mathcal{M}^{\pm}_{\lambda,\Lambda}\big(D^2 u+g(u)\nabla u\otimes\nabla u\big)+f(x,u)=0	 & \mbox{in} \;\; \Omega\\
		 u=0  & \mbox{on}\;\;\partial \Omega.
	\end{cases}
\end{equation}
Here $\xi\otimes\xi$ represents the matrix $(\xi_i\xi_j)$ and $f:\overline{\Omega}\times [0,\infty)\to \mathbb R$ and $g:[0,\infty)\to [0,\infty)$ are continuous functions. Note that, for  $\lambda=\Lambda$ we have $\mathcal{M}^{\pm}_{\lambda,\Lambda}(M)=\lambda\, \mathrm{trace}(M)$ for any symmetric $n\times n$ matrix $M$. Therefore, if $\lambda=\Lambda=1$ and $g\equiv 1$, then 
$$\mathcal{M}^{\pm}_{\lambda,\Lambda}\big(D^2 u+g(u)\nabla u\otimes\nabla u\big)=\Delta u+|\nabla u|^2.$$
Thus,  \hyperlink{Pmais}{$(\mathcal{P}^{\pm}_{g})$}  reduces to  the Laplace equation with a quadratic gradient \eqref{KK-problem}. This means that  \hyperlink{Pmais}{$(\mathcal{P}^{\pm}_{g})$} represents the Pucci equation with a quadratic gradient type term corresponding to the pure equation \eqref{Pucci-eq1}. 

In this note, we present extensions  for the new class of Pucci equations  \hyperlink{Pmais}{$(\mathcal{P}^{\pm}_{g})$} of some of the results regarding existence, non-existence, uniqueness, asymptotic behavior, and Liouville type results obtained by Felmer-Sirakov \cite{Sirakov}, Felmer-Quaas \cite{Felmer}, Cutrì-Leoni \cite{Cutri}, and Felmer-Quaas-Tang \cite{FelmerTang} for the  Pucci equation \eqref{Pucci-eq1}.

Before we state our results and to precise the assumptions we are considering on the pair of functions $f,g$, we first define the following  Kazdan-Kramer type change
\begin{equation}\label{a-KK}
    \varphi_{g}(s)=\int_{0}^{s}e^{G(t)}dt, \;\; \mbox{where}\;\; G(t)=\int_{0}^{t}g(\tau)d\tau.
\end{equation}
Note that $\varphi_{g}$ is a $C^2$ diffeomorphism on $[0,\infty)$ with $\varphi_g(0)=0$ and $\varphi_g>0$ in $(0,\infty)$. Hence, we can also define the function $h:\overline{\Omega}\times[0,\infty)\to \mathbb{R}$ given by
\begin{equation}\label{f-transformed}
   h(x,s)=e^{G(\varphi^{-1}_{g}(s))}f(x,\varphi^{-1}_g(s)).
\end{equation}
\subsection{Sub- and superlinear growths} In this section we are assuming that $\Omega\subset\mathbb{R}^n$ is a bounded regular domain. Let us consider $G$, $\varphi_g$ and $h$ be given by \eqref{a-KK} and \eqref{f-transformed}. We will consider the following hypotheses on $f$ and $g$: 
\begin{enumerate}
    \item [\hypertarget{fg0}{$(fg_0)\;$}] the function $h$ is H\"{o}lder continuous on $\overline{\Omega}\times [0,\infty)$, with  $f(x,0)=0$ and $f(x,t)\ge -\gamma \varphi_g(t)$, for all $(x,t)\in \overline{\Omega}\times [0,\infty)$ and for some $\gamma\ge 0$,\\
    
    \item [\hypertarget{fg_I}{$(fg_{I})\;$}] $\displaystyle\limsup_{t\to\infty}\frac{e^{G(t)}f(x,t)}{\varphi_{g}(t)}<\mu^{+}_1<\displaystyle\liminf_{t\to 0}\frac{e^{G(t)}f(x,t)}{\varphi_{g}(t)}\le \infty$, uniformly in $x\in\overline{\Omega}$,\\

    \item [\hypertarget{fg^I}{$(fg^I)\;$}] $\displaystyle\limsup_{t\to 0}\frac{e^{G(t)}f(x,t)}{\varphi_{g}(t)}<\mu^{+}_1<\displaystyle\liminf_{t\to \infty}\frac{e^{G(t)}f(x,t)}{\varphi_{g}(t)}\le \infty$, uniformly in $x\in\overline{\Omega},$\\
\end{enumerate}
where $\mu^{+}_1$ is the first eigenvalue of the Pucci operator $\mathcal{M}^{+}_{\lambda,\Lambda}$ associated to a positive eigenfunction, see \cite{Esteban} for details. If \hyperlink{fg_I}{$(fg_{I})$} holds  we say that the pair $f,g$ has sublinear growth, while  the pair $f,g$ has superlinear growth if \hyperlink{fg^I}{$(fg^{I})$}  holds.

\begin{remark}\label{remark-gP1} Let us consider $g(s)=ms^{m-1}$ with $m>0$ a power type nonlinearity, $f_1(x,t)=t^{p}e^{-t^m}$ with $0<p<1$ and $f_2(x,t)=\nu te^{t^q-t^m}$, with $q>m$ and $0<\nu<\mu^{+}_1$. The  functions $G$ and $\varphi_g$ in \eqref{a-KK}  become
 $G(t)=t^{m}$ and $\varphi_g(t)=\int_{0}^{t}e^{\tau^m}d\tau$. It is easy to see that
\begin{equation}\label{g-powergrowth}
\left\{ \begin{aligned}    &\lim_{t\to 0}\frac{\varphi_g(t)}{t}=1\\
&\lim_{t\to \infty}\frac{\varphi_g(t)}{t}=+\infty
\\
 &  \lim_{t\to \infty}\frac{\varphi_g(t)}{e^{t^q}}=0.
\end{aligned}\right.
\end{equation}
From \eqref{g-powergrowth}, we have
\begin{equation}
\begin{aligned}
   &\limsup_{t\to \infty}\frac{e^{G(t)}f_1(x,t)}{\varphi_{g}(t)}=\limsup_{t\to \infty}\frac{t^{p}}{t}\frac{t}{\varphi_{g}(t)}=0,\\
  & \liminf_{t\to 0} \frac{e^{G(t)}f_1(x,t)}{\varphi_{g}(t)}=\liminf_{t\to 0}\frac{t^{p}}{t}\frac{t}{\varphi_{g}(t)}=+\infty.
\end{aligned}
\end{equation}
Thus, the pair $f_1,g$ has sublinear growth \hyperlink{fg_I}{$(fg_{I})$}. Analogously, \eqref{g-powergrowth} yields
\begin{equation}
\begin{aligned}
 &  \limsup_{t\to 0}\frac{e^{G(t)}f_2(x,t)}{\varphi_{g}(t)}=\limsup_{t\to 0}\frac{\nu t}{\varphi_g(t)}e^{t^q}=\nu,\\
 &  \liminf_{t\to \infty} \frac{e^{G(t)}f_2(x,t)}{\varphi_{g}(t)}=\liminf_{t\to \infty}\frac{\nu t}{\varphi_g(t)}e^{t^q}=+\infty.
\end{aligned}
\end{equation}
Consequently, the pair $f_2,g$ has superlinear growth \hyperlink{fg^I}{$(fg^{I})$}.
\end{remark}
\begin{remark} In general, if $g(s)=ms^{m-1}$ with $m>0$  and $f(x,t)=\left(\int_{0}^{t}e^{\tau^{m}}d\tau\right)^{p}e^{-t^{m}}$, then the pair $f,g$ satisfies \hyperlink{fg0}{$(fg_0)$} and \hyperlink{fg_I}{$(fg_{I})$} for $0<p<1$, while it satisfies \hyperlink{fg0}{$(fg_0)$} and  \hyperlink{fg^I}{$(fg^{I})$} if $p>1$. 
\end{remark}

Our first result concerns to the existence of solutions for the equation \hyperlink{Pmais}{$(\mathcal{P}^{+}_{g})$} under the sublinear growth condition.
\begin{theorem}\label{thm1} Suppose the pair $f,g$ satisfies \hyperlink{fg0}{$(fg_0)$}  and \hyperlink{fg_I}{$(fg_{I})$}. Then  problem \hyperlink{Pmais}{$(\mathcal{P}^{+}_{g})$}  has a positive solution. 
\end{theorem}

Following \cite{Sirakov}, to state our result for the superlinear growth, we consider the family of problems obtained from \eqref{Pucci-eq1} by using $h(x,u+t)$ instead of $f(x,u)$, for $t\ge 0$; where $h$ is defined such in \eqref{f-transformed}. Denote by  $A_t$ the set of nonnegative classical solutions for any such problem and denote $\mathcal{S}_{t}=\displaystyle\cup_{0\le s\le t}A_s$.
\begin{theorem}\label{thm2}
 Assume the pair $f,g$ satisfies  \hyperlink{fg0}{$(fg_0)$} and \hyperlink{fg^I}{$(fg^{I})$}.  Suppose that for each $t\ge0$ there exists a constant $c>0$ depending only on $t$, $\Omega$ and $h$ such that 
 \begin{equation}\label{shift-solution}
     \|u\|_{L^{\infty}}\le c,\;\;\mbox{for all}\;\; u\in \mathcal{S}_{t}.
 \end{equation}
  Then  problem \hyperlink{Pmais}{$(\mathcal{P}^{+}_{g})$}  has a positive solution.
\end{theorem}

For the choice $g\equiv 0$, Theorem~\ref{thm1} and Theorem~\ref{thm2} recover the existence results in \cite[Theorem~1.1 and Theorem~1.2]{Sirakov}.

\begin{remark}\label{ex-several} Let $a:\overline{\Omega}\to \mathbb{R}$ be  smooth and bounded between two  suitable constants. 
 \begin{enumerate}
     \item If $g\equiv 1$ and $f(x,t)=a(x)[e^{t}-1]^{p}e^{-t}$, then the par $f,g$ satisfies \hyperlink{fg0}{$(fg_0)$} and \hyperlink{fg_I}{$(fg_{I})$} for $0<p<1$, while it satisfies \hyperlink{fg0}{$(fg_0)$} and  \hyperlink{fg^I}{$(fg^{I})$} if $p>1$.
     \item If $g(t)=\dfrac{\mu}{1+t}$, with $\mu>0$ and $f(x,t)=\dfrac{a(x)}{(\mu+1)^p}[(1+t)^{\mu+1}-1]^p(1+t)^{-\mu}$, then the pair $f,g$ satisfies \hyperlink{fg0}{$(fg_0)$} and \hyperlink{fg_I}{$(fg_{I})$} for $0<p<1$,  while it satisfies \hyperlink{fg0}{$(fg_0)$} and  \hyperlink{fg^I}{$(fg^{I})$} if $p>1$. 
   \item If $g(t) = \dfrac{\sinh(t)}{\cosh(t) + 1}$ and $f(x,t)=\frac{a(x)}{1+\cosh{t}}\Big[(\sinh(t) + t)^{p}-\gamma(\sinh(t) + t)\Big]$ with $\gamma\ge 0$, then the pair \hyperlink{fg0}{$(fg_0)$} and \hyperlink{fg_I}{$(fg_{I})$} for $0<p<1$, while it satisfies \hyperlink{fg0}{$(fg_0)$} and  \hyperlink{fg^I}{$(fg^{I})$} if $p>1$. 
     \item If $g(t) = \dfrac{e^{t}(1 + t)}{t e^{t} + 1}$ and $f(x,t)=\frac{a(x)}{te^t+1}\Big[\mu\ln\big(2+t+ (t-1)e^t\big)-\gamma\big(1+t+(t-1)e^t\big)\Big]$ with $\mu>\mu^{+}_1+\gamma$, then the pair $f,g$ satisfies \hyperlink{fg0}{$(fg_0)$} and \hyperlink{fg_I}{$(fg_{I})$}. 
 \end{enumerate}   
\end{remark}
In Theorem~\ref{thm2},  in addition to assumptions \hyperlink{fg0}{$(fg_0)$} and  \hyperlink{fg^I}{$(fg^{I})$}, we also assume the existence of the  priori bounds  \eqref{shift-solution}. As an example of pair $f,g$ that admits \eqref{shift-solution} we can take
\begin{equation}\label{ex-regular}
    g(t) = \dfrac{1}{(t + e) \ln(t + e)}\;\;\mbox{and}\quad  f(x,t)=\frac{(t+e)^p}{\ln(t+e)}\Big(\ln(t+e)-1\Big)^p,
\end{equation}
 for $1<p<p^{s}_{+}$, where $p^{s}_{+}$ is defined in \eqref{p+critical} below. In fact, for this pair, the function $h$ in \eqref{f-transformed} becomes $h(x,s)=s^{p}$ for which we have the priori bounds \eqref{shift-solution} due to \cite[Proposition~4.3]{Sirakov}. For more results on the priori bounds for fully nonlinear equations in the context of Pucci’s extremal operators we recommend \cite{Sirakov,Cabre,Sirakov-Nornberg} and the references quoted therein.
\subsection{Liouville type theorems and radially symmetric solutions}
The aim of this section is to establish Liouville type results for the Pucci equations with a gradient type term  posed on $\mathbb {R}^n$, $n\ge 3$ and radially symmetric solutions on balls $B_R$, that is, 
\begin{equation}\nonumber
   \hypertarget{LP}{(\mathcal{L}^{\pm}_{g})}\;\;\;\;\;\;\quad \mathcal{M}^{\pm}_{\lambda,\Lambda}\big(D^2 u+g(u)\nabla u\otimes\nabla u\big)+f(u)=0,\;\;\; u\ge 0\;\; \mbox{in}\;\; \Omega,
\end{equation}
where $\Omega=B_{R}\subset\mathbb{R}^n$ is the ball of radius $0<R\le \infty$ and  the pair $f,g$ satisfies suitable hypotheses. First, we suppose that 
\begin{equation}\label{fg-L}
\Omega=\mathbb{R}^n\;\;\mbox{and}\;\; f(t)=\varphi^{p}_g(t) e^{-G(t)},\; t\ge 0,
\end{equation}
 where  $\varphi_g$ and $G$ are given by \eqref{a-KK}. The question is to determine the range of $p>1$ for which  \hyperlink{LP}{$(\mathcal{L}^{\pm}_{g})$} does not admit non-trivial solutions. Following \cite{Cutri,Felmer}, we define  the dimension-like numbers
\begin{equation}\nonumber
    \widetilde{N}_{+}=\frac{\lambda}{\Lambda}(n-1)+1\;\;\mbox{and}\;\;  \widetilde{N}_{-}=\frac{\Lambda}{\lambda}(n-1)+1
\end{equation}
and the exponents $(\widetilde{N}_{+}>2)$
\begin{equation}\label{p+critical}
\begin{aligned}
   p^{s}_{+}&=\frac{\widetilde{N}_{+}}{\widetilde{N}_{+}-2}\quad \mbox{and}\quad p^{p}_{+}=\frac{\widetilde{N}_{+}+2}{\widetilde{N}_{+}-2}\\
   p^{s}_{-}&=\frac{\widetilde{N}_{-}}{\widetilde{N}_{-}-2}\quad \mbox{and} \quad p^{o}_{-}=\frac{\widetilde{N}_{-}+2}{\widetilde{N}_{-}-2}.
\end{aligned}
\end{equation}
In addition, let us denote by
\begin{align*}
p^{*}_n=2^*-1=\frac{n+2}{n-2}
\end{align*}
the critical Sobolev exponent, see for instance  \cite{CGS}. With this notation, we have the first result.
\begin{theorem}\label{LPCutri} Suppose \eqref{fg-L} holds.  If $1<p\le p^{s}_{+}$, then \hyperlink{LP}{$(\mathcal{L}^{+}_{g})$} does not admit non-trivial solution. Similarly, if  $1<p\le p^{s}_{-}$, the equation  \hyperlink{LP}{$(\mathcal{L}^{-}_{g})$} has no non-trivial solution.
\end{theorem}
Inspired by Felmer and Quaas \cite{Felmer}, in the next we are interested to  inquire about  \hyperlink{LP}{$(\mathcal{L}^{\pm}_{g})$} in the case of radially symmetric solutions.
We start with the following classification for a possible radial solutions $u=u(r)$ of Pucci equation  \hyperlink{LP}{$(\mathcal{L}^{\pm}_{g})$} with $\Omega=\mathbb{R}^n$.
\begin{enumerate}
    \item [$(i)$] $u$ is a $g$-pseudo-slow decaying solution if there exist constants $0<C_1<C_2$ such that
    \begin{equation}\nonumber
        C_1=\liminf_{r\to\infty}r^{\alpha}\varphi_{g}(u(r))<\limsup_{r\to\infty}r^{\alpha}\varphi_{g}(u(r))=C_2
    \end{equation}
    \item [$(ii)$] $u$ is a $g$-slow decaying solution if there exists $c_*>0$ such that
    \begin{equation}\nonumber
        \lim_{r\to\infty}r^{\alpha}\varphi_{g}(u(r))=c_{*}
    \end{equation}
    \item [$(iii)$]\,   $u$ is a $g$-fast decaying solution if there exists $C>0$ such that
    \begin{equation}\nonumber
        \lim_{r\to\infty}r^{\widetilde{N}-2}\varphi_{g}(u(r))=C
    \end{equation}
\end{enumerate}
where $\varphi_g$ is given by \eqref{a-KK}, $\alpha=2/(p-1)$ and $\widetilde{N}=\widetilde{N}_{+}$ or $\widetilde{N}=\widetilde{N}_{-}$, according if  \hyperlink{LP}{$(\mathcal{L}^{+}_{g})$} or  \hyperlink{LP}{$(\mathcal{L}^{-}_{g})$} is considered.

\begin{remark}\label{0-slowXg-slow} If $g\equiv 0$,  we have $\varphi_g(s)=s$ and thus the classification $(i)$, $(ii)$ and $(iii)$ coincides with that proposed by Felmer and Quaas \cite[Definition~1.1]{Felmer}. Accordingly, we will use  the expressions pseudo-slow, slow  and fast instead of $0$-pseudo-slow, $0$-slow  and $0$-fast. 
\end{remark}
\begin{theorem}\label{LP+Felmer}  Suppose that  \eqref{fg-L} and $\widetilde{N}_{+}>2$ hold. Then, there exists a critical exponent $p^{*}_{+}$ such that $\max\{p^{s}_{+}, p^{*}_{n}\}<p^{*}_{+}<p^{p}_{+}$  satisfying:  
\begin{enumerate}
    \item [$(i)$] If $1<p<p^{*}_{+}$, then \hyperlink{LP}{$(\mathcal{L}^{+}_{g})$} does not admit non-trivial radial solutions.
    \item [$(ii)$] If $p=p^{*}_{+}$, then \hyperlink{LP}{$(\mathcal{L}^{+}_{g})$} has a unique $g$-fast decaying radial solution.
    \item [$(iii)$]\, If $p^{*}_{+}<p\le p^{p}_{+}$, then \hyperlink{LP}{$(\mathcal{L}^{+}_{g})$} has a unique $g$-pseudo-slow decaying radial solution.
    \item [$(iv)$]\, If $p>p^{p}_{+}$, then \hyperlink{LP}{$(\mathcal{L}^{+}_{g})$} has a unique $g$-slow decaying radial solution.
\end{enumerate}
In $(ii)$, $(iii)$ and $(iv)$  the uniqueness is understood up to scaling.
\end{theorem}
An analogous result to \hyperlink{LP}{$(\mathcal{L}^{-}_{g})$} can be established,  which we include below for completeness.
\begin{theorem}\label{LP-Felmer}  Suppose that  \eqref{fg-L}  holds. Then, there exists a critical exponent $p^{*}_{-}$ such that $1<p^{s}_{-}<p^{*}_{-}<p^{p}_{-}:=p^{*}_{n}$ and $p^{o}_{-}<p^{*}_{-}<p^{p}_{-}$ satisfying:  
\begin{enumerate}
    \item [$(i)$] If $1<p<p^{*}_{-}$, then \hyperlink{LP}{$(\mathcal{L}^{-}_{g})$} does not admit non-trivial radial solutions.
    \item [$(ii)$] If $p=p^{*}_{-}$, then \hyperlink{LP}{$(\mathcal{L}^{-}_{g})$} has a unique $g$-fast decaying radial solution.
    \item [$(iii)$]\, If $p^{*}_{-}<p\le p^{p}_{-}$, then \hyperlink{LP}{$(\mathcal{L}^{-}_{g})$} has a unique radial solution, which is a $g$-slow decaying or a $g$-pseudo-slow decaying solution. 
    \item [$(iv)$]\, If $p>p^{p}_{-}$, then \hyperlink{LP}{$(\mathcal{L}^{-}_{g})$} has a unique $g$-slow decaying radial solution.
\end{enumerate}
In $(ii)$, $(iii)$ and $(iv)$  the uniqueness is understood up to scaling.
\end{theorem}
We note that Theorem~\ref{LP+Felmer} and Theorem~\ref{LP-Felmer} extend the analysis in \cite[Theorem~1.1 and Theorem~1.2]{Felmer} for $g\not\equiv 0$.
\begin{remark} For $g(t) = \dfrac{2t}{1 + t^2}$, we obtain that $f(t) = \dfrac{1}{1+t^ 2}\left(t + \dfrac{t^3}{3}\right)^{p}$ is an example for Theorems \ref{LPCutri}, \ref{LP+Felmer} and \ref{LP-Felmer}. 
\end{remark}
Next, we are supposing  the domain $\Omega$ and the pair $f,g$ satisfy
\begin{equation}\label{fg-BR}
 \Omega=B_R \;\;\mbox{and}\;\; f(t)=-\gamma \varphi_{g}(t)e^{-G(t)}+\psi(t), \; t\ge 0,
\end{equation}
where  $\varphi_g$ and $G$ are given by \eqref{a-KK} and $\psi:[0,\infty)\to [0,\infty)$
satisfies
\begin{enumerate}
    \item [\hypertarget{psi0}{$(\psi_0)\;$}] $\psi\in C([0,\infty))$  and $t\mapsto e^{G(\varphi^{-1}_g(t))}\psi(\varphi_g^{-1}(t))$ is locally Lipschitz on $[0,\infty)$
    \item [\hypertarget{psi1}{$(\psi_1)\;$}]  there exists $1<p<p^{*}_{\pm}$ such that
    \begin{align*}
    \lim_{t\to\infty}\frac{e^{G(t)}\psi(t)}{\varphi_{g}^p(t)}=C^*,\;\;\mbox{for some constant}\;\; C^*>0
    \end{align*}
    \item [\hypertarget{psi2}{$(\psi_2)\;$}]  there exists $c^*\ge0$ such that $c^*-\gamma<\mu^{\pm}_{1}$ and 
     \begin{align*}
    \lim_{t\to0}\frac{e^{G(t)}\psi(t)}{\varphi_{g}(t)}=c^*
    \end{align*}
    where $\mu^{\pm}_{1}$ represents the first eigenvalue  for $\mathcal{M}^{\pm}_{\lambda,\Lambda}$ in $B_R$.
\end{enumerate}
\begin{remark} 
The function $f(t)=(-\gamma+\mu +\sinh^{p-1} t)\tanh t$ with $\gamma\ge 0$, $1<p<p^*_{\pm}$ and $0\le \mu<\mu^{\pm}_1$ satisfies  \hyperlink{psi0}{$(\psi_0)$}-\hyperlink{psi0}{$(\psi_2)$} with $g(t)=\tanh t$ and $\psi(t)=\mu\tanh t+\tanh t \sinh^{p-1} t$.
\end{remark}
\begin{theorem}\label{thm-radial1} Assume $B_R$  is a finite ball and \eqref{fg-BR} with $\psi$ satisfying \hyperlink{psi0}{$(\psi_0)$}-\hyperlink{psi0}{$(\psi_2)$}. Then there exists a positive radially symmetric $u\in C^2$ of  \hyperlink{LP}{$(\mathcal{L}^{\pm}_{g})$}  with $u=0$ on $\partial B_R$.  
\end{theorem}
By choosing a prototype $\psi(t)=\varphi^{p}_g(t)e^{-G(t)}$ and $\gamma=1$, we are also able to prove the following uniqueness result.
\begin{theorem}\label{thm-radial2} Assume $f(t)=[-\varphi_g(t)+\varphi^{p}_{g}(t)]e^{-G(t)}$ with $1<p<p^{*}_{\pm}$. Then,
\begin{enumerate}
\item If $\Omega=B_R$ is a finite ball,   \hyperlink{LP}{$(\mathcal{L}^{\pm}_{g})$} admits exactly one positive radial solution with $u=0$ on $\partial B_R$
\item If $\Omega=\mathbb{R}^n$,  \hyperlink{LP}{$(\mathcal{L}^{\pm}_{g})$} admits exactly one positive radial solution with $\lim u(|x|)\to 0$ as $|x|\to \infty$.
\end{enumerate}  
\end{theorem}
Theorem~\ref{thm-radial2}  extends for the Pucci equations with a gradient type term, the  uniqueness result in  \cite[Theorem~1.1]{FelmerTang}.
\begin{remark} For $g(t)=\dfrac{e^t}{1+e^t}$, we get that $ f(t)=\dfrac{2}{1+e^t}\Big[\Big(\dfrac{e^{t} + t-1}{2}\Big)^p-\Big(\dfrac{e^{t} + t-1}{2}\Big)\Big]$ with $1<p<p^{*}_{\pm}$
is a typical example for Theorem~\ref{thm-radial2}.
\end{remark}

\section{Kazdan-Kramer change}
Let $M$ and $N$ be two symmetric matrices. The following basic properties of $\mathcal{M}^{\pm}=\mathcal{M}^{\pm}_{\lambda,\Lambda}$ can be founded in \cite[Lemma~2.10]{Cabre}.
\begin{enumerate}

    \item [\hypertarget{1P}{$(1)$}] $\mathcal{M}^{\pm}(\alpha M)=\alpha\mathcal{M}^{\pm}(M)$ if $\alpha\ge 0$
    \item [\hypertarget{2P}{$(2)$}] $\mathcal{M}^{+}(M)+\mathcal{M}^{-}(N)\le \mathcal{M}^{+}(M+N)\le \mathcal{M}^{+}(M)+\mathcal{M}^{+}(N)$
    \item [\hypertarget{3P}{$(3)$}] $\mathcal{M}^{-}(M)+\mathcal{M}^{-}(N)\le \mathcal{M}^{-}(M+N)\le \mathcal{M}^{+}(M)+\mathcal{M}^{-}(N)$.
\end{enumerate}

In the next two results $\Omega$ represents a general open set in $\mathbb{R}^n$.
   \begin{lemma} \label{lemma1} Let  $u\in C^2(\Omega)$ and $I\subset \mathbb{R}$ be an interval with $u(\Omega)\subset I$ and $\varphi\in C^{2}(I)$ with  $\varphi^{\prime}> 0$, $\varphi^{\prime\prime}\ge 0$ and set  $v=\varphi(u)$. Then
   \begin{equation}\label{=M+}
    (\varphi^{-1})^{\prime}(v)\mathcal{M}^{\pm}( D^{2}v) = \mathcal{M}^{\pm}\Big(D^2u+\frac{\varphi^{\prime\prime}(u)}{\varphi^\prime(u)}\nabla u\otimes\nabla u\Big).
    \end{equation}
In addition, 
\begin{equation}\label{M+}
   \mathcal{M}^{+}( D^2u)+\frac{\varphi^{\prime\prime}(u)}{\varphi^\prime(u)}\mathcal{M}^{-}(\nabla u\otimes\nabla u)   \le (\varphi^{-1})^{\prime}(v)\mathcal{M}^{+}( D^{2}v) \le   \mathcal{M}^{+}( D^2u)+\frac{\varphi^{\prime\prime}(u)}{\varphi^\prime(u)}\mathcal{M}^{+}(\nabla u\otimes\nabla u)
    \end{equation}
    and
    \begin{equation}\label{M-}
         \mathcal{M}^{-}( D^2u)+\frac{\varphi^{\prime\prime}(u)}{\varphi^\prime(u)}\mathcal{M}^{-}(\nabla u\otimes\nabla u)   \le  (\varphi^{-1})^{\prime}(v)\mathcal{M}^{-}( D^{2}v)\! \le  \! \mathcal{M}^{-}( D^2u)+\frac{\varphi^{\prime\prime}(u)}{\varphi^\prime(u)}\mathcal{M}^{+}(\nabla u\otimes\nabla u).
    \end{equation}
\end{lemma}
\begin{proof}
   We have $v_{x_i}=\varphi^{\prime}(u)u_{x_i}$ and $v_{x_ix_j}=\varphi^{\prime}(u)u_{x_ix_j}+\varphi^{\prime\prime}(u)u_{x_i}u_{x_j}$.  Hence
    \begin{equation}\label{comp-Hessian}
        D^{2}v= \varphi^\prime(u) D^2u+\varphi^{\prime\prime}(u)\nabla u\otimes\nabla u.
    \end{equation}
Hence,  from the basic property of \hyperlink{1P}{$(1)$}  above,  we can write
    \begin{equation}\label{=ineq1}
    \begin{aligned}
  \mathcal{M}^{\pm}( D^{2}v) &=  \mathcal{M}^{\pm}\big(\varphi^\prime(u) D^2u+\varphi^{\prime\prime}(u)\nabla u\otimes\nabla u\big)\\
  &= \varphi^\prime(u)\mathcal{M}^{\pm}\Big(D^2u+\frac{\varphi^{\prime\prime}(u)}{\varphi^{\prime}(u)}\nabla u\otimes\nabla u\Big)
  \end{aligned}
    \end{equation}
Noticing that $(\varphi^{-1})^{\prime}(v)=1/\varphi^{\prime}(u)$, \eqref{=ineq1} yields \eqref{=M+}. The inequalities \eqref{M+} and \eqref{M-} follow from \eqref{=ineq1}, and from the properties \hyperlink{2P}{$(2)$} and \hyperlink{3P}{$(3)$} above.
\end{proof}
The next result shows that the gradient type term $\nabla u\otimes\nabla u$ is invariant to the Pucci extremals operators by the Kazdan-Kramer change \eqref{a-KK}. 
\begin{lemma}\label{via-lemma}
Let $f:\Omega\times \mathbb [0,\infty)\to \mathbb {R}$ and $g:[0,\infty)\to [0,\infty)$ be continuous functions and let $\varphi_{g}$ and $G$ given by \eqref{a-KK}. 
Suppose that  $v\in C^{2}(\Omega)$, $v\ge 0$ solves the equation 
\begin{equation}\label{Pucci-eq1-transformed}
	\mathcal{M}^{\pm}_{\lambda,\Lambda}(D^2 v)+e^{G(\varphi^{-1}_g(v))}f(x, \varphi^{-1}_{g}(v))=0 \;\; \mbox{in} \;\; \Omega.
\end{equation}
Then, the function $u=\varphi^{-1}_{g}(v)$ is a non-negative solution to the equation 
\begin{equation}\label{Pucci-com-grad}
    \mathcal{M}^{\pm}_{\lambda,\Lambda}\big(D^2 u+g(u)\nabla u\otimes\nabla u\big)+f(x,u)=0	 \;\; \mbox{in} \;\; \Omega.
\end{equation}
Reciprocally, if $u\in C^{2}(\Omega)$, $u\ge0$ solves  \eqref{Pucci-com-grad}, then $v=\varphi_{g}(u)$ is a solution to \eqref{Pucci-eq1-transformed}. 
\end{lemma}
\begin{proof}
As mentioned before,  $\varphi_{g}$ is a $C^2$ diffeomorphism on $[0,\infty)$ such that $\varphi_{g}(0)=0$, $\varphi^{\prime}_{g}>0$, $\varphi^{\prime\prime}_{g}\ge 0$, and ${\varphi^{\prime\prime}_{g}}/{\varphi^{\prime}_{g}}=g$ on $[0,\infty)$. In addition, $(\varphi^{-1}_{g})^{\prime}(t)= e^{-G(\varphi^{-1}_g(t))},\, t\ge 0$. Thus, if $v\ge0$ and $u=\varphi^{-1}_{g}(v)$, the identity \eqref{=M+} yields
\begin{equation}\nonumber
\begin{aligned}
    \mathcal{M}^{\pm}_{\lambda,\Lambda}\Big(D^2u+g(u)\nabla u\otimes\nabla u\Big)+f(x,u) &=(\varphi^{-1}_{g})^{\prime}(v)\mathcal{M}^{\pm}_{\lambda,\Lambda}( D^{2}v)+ f(x,\varphi^{-1}_{g}(v))\\
    &= e^{-G(\varphi^{-1}_g(v))}\left[\mathcal{M}^{\pm}_{\lambda,\Lambda}( D^{2}v)+ e^{G(\varphi^{-1}_g(v))}f(x,\varphi^{-1}_{g}(v))\right].
\end{aligned}
\end{equation}
 Thus,  $v$ solves \eqref{Pucci-eq1-transformed} if and only if $u$ solves \eqref{Pucci-com-grad}.
\end{proof}
\section{Existence for sub- and superlinear growths}
Let $h:\overline{\Omega}\times[0,\infty)\to \mathbb{R}$ be given by \eqref{f-transformed}, i.e, $h(x,s)=e^{G(\varphi^{-1}_{g}(s))}f(x,\varphi^{-1}_g(s))$. From \hyperlink{fg0}{$(fg_0)$} we can see that $h$ satisfies the assumption:
\begin{enumerate}
\item [$(h_0)$]\, $h$ is a H\"{o}lder continuous on $\overline{\Omega}\times [0,\infty)$, $h(x,0)=0$ and $h(x,s)\ge -\gamma s$, for all $(x,s)\in \overline{\Omega}\times [0,\infty)$ and for some $\gamma\ge 0$.
\end{enumerate}
In addition, by using the change $t=\varphi^{-1}_g(s)$ we have that
\begin{equation}\label{limINf}
     \liminf_{s\to \infty} \frac{h(x,s)}{s}=\liminf_{t\to \infty} \frac{e^{G(t)}f(x,t)}{\varphi_{g}(t)}\;\;\mbox{and}\;\; \liminf_{s\to 0} \frac{h(x,s)}{s}=\liminf_{t\to 0} \frac{e^{G(t)}f(x,t)}{\varphi_{g}(t)}
\end{equation}
and
\begin{equation}\label{limSUP}
    \limsup_{s\to \infty} \frac{h(x,s)}{s}=\limsup_{t\to \infty} \frac{e^{G(t)}f(x,t)}{\varphi_{g}(t)}\;\;\mbox{and}\;\; \limsup_{s\to 0} \frac{h(x,s)}{s}=\limsup_{t\to 0} \frac{e^{G(t)}f(x,t)}{\varphi_{g}(t)}.
\end{equation}
\begin{proof}[Proof of Theorem~\ref{thm1}] By using \hyperlink{fg_I}{$(fg_{I})$}, \eqref{limINf} and \eqref{limSUP} we have that $h$ satisfies the sublinear condition
\begin{enumerate}
\item [$(H_0)$]\, $\displaystyle\limsup_{s\to \infty} \frac{h(x,s)}{s}<\mu^{+}_1<\displaystyle\liminf_{s\to 0} \frac{h(x,s)}{s}\le \infty$, uniformly in $x\in\overline{\Omega}$.
\end{enumerate}
Since $h$ satisfies $(h_0)$ and $(H_0)$,  \cite[Theorem~1.1]{Sirakov} ensures the existence of a positive classical solution $v$ to \eqref{Pucci-eq1-transformed} with the boundary condition $v=0$ on $\partial\Omega$. According to Lemma~\ref{via-lemma}, the function $u=\varphi^{-1}_{g}(v)$ is a solution to \eqref{Pucci-com-grad} and since $\varphi^{-1}_{g}(0)=0$, we have $u=0$ on $\partial\Omega$. Thus,  $u$ is a positive solution for \hyperlink{Pmais}{$(\mathcal{P}^{+}_{g})$}.
\end{proof}

\begin{proof}[Proof of Theorem~\ref{thm2}] We already have that $h$ satisfies the $(h_0)$. In addition, from   \hyperlink{fg^I}{$(fg^{I})$}, \eqref{limINf} and \eqref{limSUP} we obtain that $h$ satisfies the superlinear condition
\begin{enumerate}
\item [$(H^0)$]\, $\displaystyle\limsup_{s\to 0} \frac{h(x,s)}{s}<\mu^{+}_1<\displaystyle\liminf_{s\to \infty} \frac{h(x,s)}{s}\le \infty$, uniformly in $x\in\overline{\Omega}$.
\end{enumerate}
For each $t\ge0$, let us consider the following family of  Pucci equations
\begin{equation}\nonumber
(\mathcal{P}_t)	\quad \quad \begin{cases}
	\mathcal{M}^{+}_{\lambda,\Lambda}(D^2 v)+h(x,v+t)=0	 & \mbox{in} \;\; \Omega\\
		 v=0  & \mbox{on}\;\;\partial \Omega.	
	\end{cases}
\end{equation}
 Denote by $A_t$ the set of nonnegative classical solutions of  $(\mathcal{P}_t)$ and set $\mathcal{S}_{t}=\displaystyle\cup_{0\le s\le t}A_s$.
Then \eqref{shift-solution} ensures $\|v\|_{L^{\infty}}\le c,\;\;\mbox{for all}\;\; v\in \mathcal{S}_{t}$,
 for some $c>0$ depending only on $t, \Omega$ and $h$. So, from  \cite[Theorem~1.2]{Sirakov}, the equation  \eqref{Pucci-eq1-transformed} admits a classical positive solution $v$. Now, the result follows from  Lemma~\ref{via-lemma} by choosing the solution $u=\varphi^{-1}_g(v)$.
 \end{proof}
 \section{Liouville type results and radially symmetric solutions}
 In this section we deal with Liouville type results and the existence and asymptotic behavior of positive radial solutions for the Pucci equations stated in Theorem~\ref{LPCutri}, Theorem~\ref{LP+Felmer}, Theorem~\ref{LP-Felmer}, Theorem~\ref{thm-radial1} and Theorem~\ref{thm-radial2}. Our main argument relies on results previously obtained by  Cutr\`{i}-Leoni  \cite{Cutri}, Felmer-Quaas \cite{Felmer} and Felmer-Quaas-Tang \cite{FelmerTang} for Pucci equations without the gradient type term $\nabla u\otimes\nabla u$. Let us recall the corresponding transformed function $h:[0,\infty)\to \mathbb{R}$ given by
\begin{equation}\label{L-f-transformed}
   h(s)=e^{G(\varphi^{-1}_{g}(s))}f(\varphi^{-1}_g(s))
\end{equation}
where $\varphi_g$ and  $G$ are given by \eqref{a-KK}. The following Lemma summarizes some useful properties.
\begin{lemma}\label{Autovia-lemma}
Let $f:[0,\infty)\to \mathbb {R}$ and $g:[0,\infty)\to [0,\infty)$ be continuous functions and let $h$ be given by \eqref{L-f-transformed}.  Let us consider the Pucci equations
\begin{equation}\label{AutoPucci-eq1-transformed}
\mathcal{M}^{\pm}_{\lambda,\Lambda}(D^2 v)+h(v)=0 ,\;\;\; v\ge 0\;\; \mbox{in} \;\; \mathbb{R}^n	
\end{equation}
and 
\begin{equation}\label{AutoPucci-com-grad}
    \mathcal{M}^{\pm}_{\lambda,\Lambda}\big(D^2 u+g(u)\nabla u\otimes\nabla u\big)+f(u)=0,\;\;\; u\ge 0	 \;\; \mbox{in} \;\; \mathbb{R}^n.
\end{equation}
Then
\begin{enumerate}
\item[$(a)$]  $v$ solves of  \eqref{AutoPucci-eq1-transformed} if and only if $u=\varphi^{-1}_{g}(v)$ solves \eqref{AutoPucci-com-grad}.
\item [$(b)$] $v$ is a radial pseudo-slow decaying solution of \eqref{AutoPucci-eq1-transformed} if and only if $u=\varphi^{-1}_{g}(v)$ is a $g$-pseudo-slow decaying solution of  \eqref{AutoPucci-com-grad}.
\item [$(c)$] $v$ is a radial slow decaying solution of \eqref{AutoPucci-eq1-transformed} if and only if $u=\varphi^{-1}_{g}(v)$ is a $g$-slow decaying solution of  \eqref{AutoPucci-com-grad}.
\item [$(d)$] $v$ is a radial fast decaying solution of \eqref{AutoPucci-eq1-transformed} if and only if $u=\varphi^{-1}_{g}(v)$ is a $g$-fast decaying solution of  \eqref{AutoPucci-com-grad}.
\end{enumerate}
\end{lemma}
\begin{proof} The item $(a)$ follows from  Lemma~\ref{via-lemma}. It is clear that $v=v(r)$ a symmetric radially function if and only if $u(r)=\varphi^{-1}_g(v(r))$ is a symmetric radially function. Hence, the identity $r^{\alpha}v(r)=r^{\alpha}\varphi_g(u(r))$ yields $(b)$, $(c)$ and $(d)$.
\end{proof}
Now, under the assumption \eqref{fg-L}, the transformed function $h$ in  \eqref{L-f-transformed} becomes $h(s)=s^{p}$, $s\ge 0$. Thus,  the Pucci equation \eqref{Pucci-eq1-transformed} turns 
 \begin{equation}\label{CutriPucci-transformed}
	\mathcal{M}^{\pm}_{\lambda,\Lambda}(D^2 v)+v^{p}=0,\;\;\; v\ge 0 \;\; \mbox{in} \;\; \mathbb{R}^n.
\end{equation}
\begin{proof}[Proof of Theorem~\ref{LPCutri}] 
According to \cite{Cutri} (see also \cite{Felmer} and \cite[Theorem~1.3]{Sirakov}),  equation \eqref{CutriPucci-transformed} with  the operator $\mathcal{M}^{+}_{\lambda,\Lambda}$ (or $\mathcal{M}^{-}_{\lambda,\Lambda}$)  does not admit non-trivial (viscosity) solutions for  $1<p\le p^{s}_{+}$ (or $1<p\le p^{s}_{-}$).  Due to the ellipticity  of the Pucci operators $\mathcal{M}^{\pm}_{\lambda,\Lambda}$, the same conclusion remains valid for classical solutions instead of viscosity solutions. Thus,  from Lemma~\ref{Autovia-lemma}, item $(a)$  we can conclude Theorem~\ref{LPCutri}.
\end{proof}

\begin{proof}[Proof of Theorem~\ref{LP+Felmer}]  From \cite[Theorem~1.1]{Felmer},  equation \eqref{CutriPucci-transformed} with  the operator $\mathcal{M}^{+}_{\lambda,\Lambda}$ satisfies:
\begin{enumerate}
    \item [$(i)$]  it does not admit non-trivial radial solutions, if $1<p<p^{*}_{+}$
    \item  [$(ii)$]  it has a unique fast decaying radial solution, if $p=p^{*}_{+}$
    \item [$(iii)$]\,  it has a unique pseudo-slow decaying radial solution, if $p^{*}_{+}<p\le p^{p}_{+}$
    \item [$(iv)$]\,   it  has a unique slow decaying radial solution, if $p>p^{p}_{+}$.
\end{enumerate}
From this, Lemma~\ref{Autovia-lemma} ensures the result.
\end{proof}

\begin{proof}[Proof of Theorem~\ref{LP-Felmer}] Follows  analogously  to  Theorem~\ref{LP+Felmer} by combining  \cite[Theorem~1.2]{Felmer} with Lemma~\ref{Autovia-lemma}.
\end{proof}

\begin{proof}[Proof of Theorem~\ref{thm-radial1}] Under the assumption \eqref{fg-BR} the function $h$ in \eqref{L-f-transformed} becomes
\begin{equation*}
\begin{aligned}
h(s)=e^{G(\varphi^{-1}_g(s))}f(\varphi^{-1}_g(s))&=e^{G(\varphi^{-1}_g(s))}[-\gamma s e^{-G(\varphi^{-1}_{g}(s))}+\psi(\varphi_{g}^{-1}(s))]\\
&= -\gamma s+ \overline{h}(s)
\end{aligned}
\end{equation*}
where
\begin{equation}\label{H-radial1}
\overline{h}(s)=e^{G(\varphi^{-1}_g(s))}\psi(\varphi_{g}^{-1}(s)).
\end{equation}
Note that the change $t=\varphi^{-1}(s)$ yields
\begin{equation}\label{hpsi-limits-infty}
\begin{aligned}
\lim_{s\to \infty}\frac{\overline{h}(s)}{s^p}=\lim_{s\to \infty}\frac{e^{G(\varphi^{-1}_g(s))}\psi(\varphi_{g}^{-1}(s))}{s^{p}}
=\lim_{t\to \infty}\frac{e^{G(t)}\psi(t)}{\varphi^{p}_{g}(t)}
\end{aligned}
\end{equation}
and 
\begin{equation}\label{hpsi-limits-0}
\begin{aligned}
\lim_{s\to 0}\frac{\overline{h}(s)}{s}=\lim_{s\to 0}\frac{e^{G(\varphi^{-1}_g(s))}\psi(\varphi_{g}^{-1}(s))}{s}
=\lim_{t\to 0}\frac{e^{G(t)}\psi(t)}{\varphi_g(t)}.
\end{aligned}
\end{equation}
Thus, by combining \eqref{hpsi-limits-infty} and \eqref{hpsi-limits-0} with the assumptions \hyperlink{psi0}{$(\psi_0)$}, \hyperlink{psi1}{$(\psi_1)$} and \hyperlink{psi2}{$(\psi_2)$} we can see that $\overline{h}$ satisfies
\begin{enumerate}
    \item [\hypertarget{H0}{$(\overline{h}_0)$}]\, $\overline{h}\in C([0,\infty))$  and is locally Lipschitz on $[0,\infty)$
    \item  [\hypertarget{H1}{$(\overline{h}_1)$}]\,  $\overline{h}(s)\ge 0$ and there exists $1<p<p^{*}_{\pm}$ such that
    \begin{align*}
   \lim_{s\to \infty}\frac{\overline{h}(s)}{s^p}=C^*,\;\;\mbox{for some constant}\;\; C^*>0
    \end{align*}
     \item  [\hypertarget{H2}{$(\overline{h}_2)$}] \, there exists $c^*\ge0$ such that $c^*-\gamma<\mu^{\pm}_{1}$ and 
     \begin{align*}
    \lim_{s\to 0}\frac{\overline{h}(s)}{s}=c^*
    \end{align*}
    where $\mu^{\pm}_{1}$ represents the first eigenvalue  for $\mathcal{M}^{\pm}_{\lambda,\Lambda}$ in $B_R$.
\end{enumerate}
Under the hypotheses \hyperlink{H0}{$(\overline{h}_0)$},  \hyperlink{H1}{$(\overline{h}_1)$} and \hyperlink{H2}{$(\overline{h}_2)$}, \cite[Theorem~1.2]{FelmerQL} ensure the existence of a positive radially symmetric $C^2$ solution $v$ for the Pucci equation
\begin{equation}\nonumber
	\begin{cases}
	\mathcal{M}^{\pm}_{\lambda,\Lambda}(D^2 v)-\gamma v+\overline{h}(v)=0	 & \mbox{in} \;\; B_R\\
		 v=0  & \mbox{on}\;\;\partial B_R.	
	\end{cases}
\end{equation}
According to Lemma~\ref{via-lemma}, the function $u=\varphi^{-1}(v)$ is positive radially symmetric  $C^2$ solution of  \hyperlink{LP}{$(\mathcal{L}^{\pm}_{g})$}  with $u=0$ on $\partial B_R$.  
\end{proof}
\begin{proof}[Proof of Theorem~\ref{thm-radial2}] In the case $f(t)=[-\varphi_g(t)+\varphi^{p}_{g}(t)]e^{-G(t)}$   the function $h$ in \eqref{L-f-transformed} simply reduces to $h(s)=-s+s^p$. Hence, the result follows by combining \cite[Theorem~1.1]{FelmerTang} with Lemma~\ref{via-lemma}.
\end{proof}
\subsection*{Funding}
\begin{sloppypar}
 This work was partially supported by Conselho Nacional de Desenvolvimento Cient\'{i}fico e Tecnol\'{o}gico (\url{http://dx.doi.org/10.13039/501100003593}), Grants 312340/2021-4, 409764/2023-0, 443594/2023-6, 429285/2016-7, 309491/2021-5, CAPES MATH AMSUD grant 88887.878894/2023-00,  Fondo Nacional de Desarrollo Cient\'{i}fico y Tecnol\'{o}gico 
(\url{http://dx.doi.org/10.13039/501100002850}), Grants  1181125, 1161635 and 1171691, and
Funda\c{c}\~{a}o de Apoio \`{a} Pesquisa do Estado da Para\'{i}ba (\url{http://dx.doi.org/10.13039/501100005669}), Grant 3034/2021.
\end{sloppypar}

\end{document}